\documentclass{amsart}
\usepackage{amsmath,amssymb,amsthm}
\usepackage{mathbbol}
\usepackage{enumerate}
\usepackage[hidelinks]{hyperref}
\usepackage{cleveref}
\usepackage{tikz-cd}

\numberwithin{equation}{section}
\theoremstyle{plain}
\newtheorem{theorem}{Theorem}[section]
\newtheorem{proposition}[theorem]{Proposition}
\newtheorem{corollary}[theorem]{Corollary}
\newtheorem{lemma}[theorem]{Lemma}
\theoremstyle{definition}
\newtheorem{definition}[theorem]{Definition}
\newtheorem{remark}[theorem]{Remark}
\newtheorem{example}[theorem]{Example}

\newcommand\CC{\mathbf{C}}
\newcommand\ZZ{\mathbf{Z}}
\newcommand\QQ{\mathbf{Q}}
\newcommand\LL{\mathbf{L}}
\newcommand\FF{\mathbf{F}}
\newcommand\aff{\mathbf{A}}
\newcommand\proj{\mathbf{P}}
\newcommand\symgrp{\mathfrak{S}}

\newcommand\cA{\mathcal{A}}
\newcommand\cC{\mathcal{C}}
\newcommand\cL{\mathcal{L}}
\newcommand\cR{\mathcal{R}}
\newcommand\cX{\mathcal{X}}

\newcommand\sm{\mathrm{Sm}}
\newcommand\smqp{\mathrm{SmQP}}
\newcommand\snqp{\mathrm{SNQP}}
\newcommand\sch{\mathrm{Sch}}
\newcommand\qp{\mathrm{QP}}
\newcommand\topcat{\mathrm{Top}}
\newcommand\Rad{\mathrm{Rad}}
\newcommand\cor{\mathrm{Cor}}
\newcommand\sh{\mathrm{SH}}
\newcommand\dm{\mathrm{DM}}
\newcommand\spt{\overline{\mathrm{SPT}}}
\newcommand\nis{\mathrm{Nis}}
\newcommand\rad{\mathrm{rad}}
\newcommand\tr{\mathrm{tr}}
\newcommand\op{\mathrm{op}}
\newcommand\eff{\mathrm{eff}}

\newcommand\bp{\mathrm{BP}}
\newcommand\MU{\mathrm{MU}}
\newcommand\mgl{\mathrm{MGL}}
\newcommand\hz{\mathrm{H}\ZZ}
\newcommand\hf{\mathrm{H}\FF}
\newcommand\btau{\boldsymbol{\tau}}
\newcommand\bk{\boldsymbol{k}}
\newcommand\bone{\mathbf{1}}

\newcommand\Sym{\mathrm{Sym}}
\newcommand\Conj{\mathrm{Conj}}
\DeclareMathOperator{\Spec}{Spec}
\DeclareMathOperator{\Stab}{Stab}
\DeclareMathOperator{\real}{Re}
\DeclareMathOperator{\Sq}{Sq}
\DeclareMathOperator{\CH}{CH}
\DeclareMathOperator*{\hocolim}{hocolim}
\DeclareMathOperator*{\holim}{holim}
\DeclareMathOperator*{\colim}{colim}

\author{Jiahao Hu}
\newcommand{\Addresses}{{%
  \bigskip
  \footnotesize

  \textsc{Yau Mathematical Sciences Center, Tsinghua University,
    Beijing, China 100084}\par\nopagebreak
  Email: \texttt{jiahao.hu.math@gmail.com}
}}
\date{}

\title[Motivic Steenrod algebra and Thom obstructions]{Motivic Steenrod algebra and Thom obstructions without desingularization}
\begin{document}
\maketitle
\begin{abstract}
We give a desingularization-free proof of Voevodsky's identification of bistable mod-$\ell$ motivic cohomology operations with the motivic Steenrod algebra over fields of characteristic zero.  This supplies the direct argument anticipated by Voevodsky in his study of motivic Eilenberg--MacLane spaces. 

As an application, we study Thom's topological obstructions to desingularization, which arose in his work on the Steenrod realization problem.  Without invoking desingularization, we prove that the primary and all higher Thom obstructions vanish on complex algebraic cycles. The proof uses the motivic bidegrees of the successive $k$-invariants in the Brown--Peterson tower to show that every lifting obstruction vanishes.
\end{abstract}
\tableofcontents

\section{Introduction}
In his work on the Steenrod realization problem, Thom developed topological obstructions to representing integral homology classes by smooth oriented manifolds. Applied to the fundamental classes of singular algebraic varieties, these give necessary conditions for the existence of a resolution of singularities \cite{thom1954global}. Hironaka's resolution theorem \cite{hironaka1964resolution1,hironaka1964resolution2}, or alternatively Gabber's refinement of de Jong's alterations applied prime by prime \cite[Theorem~2.1]{illusie2014gabber}, implies that these obstructions vanish on algebraic cycle classes on complex projective varieties. The starting point of this paper is instead to seek a direct explanation of this vanishing that does \textit{not} invoke resolution of singularities or alterations. We hope that understanding its mechanism may shed light on the problem of resolution in positive characteristic.

We approach these obstructions through the stronger problem of realizing homology classes by stably almost complex manifolds. This amounts to lifting them to complex bordism, represented by the spectrum $\MU$. The corresponding cohomological lifting problem is controlled, prime by prime, by the Brown--Peterson spectrum. For each prime $\ell$, Brown and Peterson construct an integral tower of fibrations over the Eilenberg--MacLane spectrum $\hz$, with Eilenberg--MacLane fibers and explicit inductively constructed $k$-invariants \cite[Theorem~1.1]{brown1966spectrum}. 

Our main observation is that this construction lifts to $\sh(\CC)$, the stable motivic homotopy category over $\CC$, and that the successive motivic $k$-invariants have bidegrees $(2d+1,d)$, with $d\ge0$. These degrees explain the vanishing of the obstructions on algebraic cycles. Indeed, motivic cohomology with supports in a closed subvariety $X$ of a smooth complex variety $M$ satisfies $H_X^{a,b}(M;\ZZ)=0$ for $a>2b$ and $b\ge0$. For a cycle of codimension $q$ supported on $X$, every successive obstruction therefore belongs to a group of the form
\begingroup
\setlength{\abovedisplayskip}{4.2pt}
\setlength{\belowdisplayskip}{4.2pt}
\setlength{\abovedisplayshortskip}{0pt}
\setlength{\belowdisplayshortskip}{2.1pt}
\[H_X^{2q+2d+1,q+d}(M;\ZZ)=0.\]
\endgroup
After Betti realization, this implies that Thom's topological obstructions also vanish. We consequently prove that every algebraic cycle class $\sigma\in H_{2n}(X(\CC);\ZZ)$, for $X$ projective, is represented by a continuous map $f:Y\to X(\CC)$ from a closed stably almost complex manifold of real dimension $2n$, with $f_*[Y]=\sigma$.

However, there is a dependence on resolution hidden in this argument. To carry out the Brown--Peterson construction motivically, we need the identification of the motivic Steenrod algebra with the algebra of all bistable mod-$\ell$ motivic cohomology operations. Voevodsky's proof uses resolution of singularities \cite[Theorem~3.49]{voevodsky2010motivic}, while the extension to $\ell$ not equal to the characteristic of the field by Hoyois, Kelly, and {\O}stv{\ae}r replaces resolution by Gabber's refinement of de Jong's alterations \cite{hoyois2017steenrod}. Therefore we must seek an approach to the identification of operations that avoids resolution of singularities and alterations.

The good news is that the dependence on resolution in Voevodsky's proof can be traced to a specific comparison theorem. Voevodsky first computes the motives of motivic Eilenberg--MacLane spaces on the category of seminormal quasi-projective schemes, using the motivic Dold--Thom theorem and symmetric powers. To transfer the computation to smooth schemes, he proves that forming reduced motives commutes with restriction to the smooth site. This is the role of Theorem~1.21 in the proof of Theorem~3.32 of \cite{voevodsky2010motivic}, and it is this comparison theorem that uses resolution of singularities. For the identification of bistable operations, the comparison is needed only for the Eilenberg--MacLane spaces occurring in the argument. Furthermore, by the motivic Dold--Thom theorem, the comparison for Eilenberg--MacLane spaces reduces to that for finite symmetric powers of $S^{2n,n}=\aff^n/(\aff^n-\{0\})$, and eventually to the quotient schemes
\begingroup
\setlength{\abovedisplayskip}{4.2pt}
\setlength{\belowdisplayskip}{4.2pt}
\setlength{\abovedisplayshortskip}{0pt}
\setlength{\belowdisplayshortskip}{2.1pt}
\[
(\aff^n)^d/\symgrp_d
\quad\text{and}\quad
\bigl((\aff^n)^d-\{0\}\bigr)/\symgrp_d,
\]
\endgroup
where $\symgrp_d$ permutes the factors.

In this paper, we establish this comparison without desingularization by directly studying the geometry of these quotients. More generally, let $G$ be a finite group acting linearly on a finite-dimensional vector space $V$ over a field in which $|G|$ is invertible, and let $U$ be the complement of a $G$-stable finite union of proper linear subspaces. We show that $V/G$ and $U/G$ admit good stratifications: each stratum is smooth and has a tubular-type neighborhood in the Nisnevich topology. The existence of such stratifications yields the comparison for finite linear quotients and hence for the Eilenberg–MacLane spaces needed in Voevodsky's argument. This successfully replaces the step of his proof that uses resolution of singularities.

To the best of our knowledge, this gives the first proof of Voevodsky's identification of bistable operations that avoids both resolution of singularities and alterations. In our comparison argument, the characteristic-zero assumption enters in two places. First, it guarantees that Maschke's theorem is applicable in our study of finite linear quotients: when the group order is invertible in the base field, every invariant linear subspace admits an equivariant complement. Second, it allows us to apply the motivic Dold--Thom theorem over $\ZZ$. 

Over a perfect field of characteristic $p>0$, the motivic Dold--Thom theorem is available after inverting $p$ \cite[Theorem~3.7]{voevodsky2010motivic}. We expect that a sufficiently detailed understanding of finite linear quotients in characteristic $p$, including those by groups whose order is divisible by $p$, would allow our comparison method to extend to the setting $\ell\neq p$. Combined with \'etale comparison in place of Betti realization \cite[Section~3.3]{hoyois2017steenrod}, this could yield a proof of the corresponding identification of bistable mod-$\ell$ operations without using alterations.

The paper is organized as follows. \Cref{sec:finite-linear-quotients} studies the geometry of finite linear quotients and constructs their good stratifications. \Cref{sec:change-of-category} proves the comparison of reduced motives for schemes admitting such stratifications. \Cref{sec:motivic-em} applies this result to symmetric powers and Eilenberg--MacLane spaces and establishes Voevodsky's identification of bistable operations over fields of characteristic zero without resolution of singularities or alterations. Finally, \Cref{sec:thom-obstructions} constructs the motivic Brown--Peterson tower and uses the bidegrees of its $k$-invariants to prove the vanishing of the primary and all higher Thom obstructions on algebraic cycle classes.

\subsection*{Acknowledgements} The author thanks Dennis Sullivan for introducing him to this problem and for numerous discussions and constant encouragement, Martin Bendersky for teaching him about the Brown--Peterson spectrum, Nanjun Yang for helpful conversations about motivic cohomology, and Santai Qu for help with algebraic geometry.
\section{Finite linear quotients}\label{sec:finite-linear-quotients} Let $k$ be a field and $G$ a finite group of order invertible in $k$. The goal of this section is to study the geometry of the scheme-theoretic quotient $V/G$ of a finite-dimensional linear space $V$ by a linear $G$-action. We show that $V/G$ admits a stratification by smooth strata and that each stratum has a tubular-type neighborhood in the Nisnevich topology.

The stratification on $V/G$ will be inherited from one on $V$, and will be constructed inductively. In each inductive step, we deal with the complement of previously constructed strata, which is a $G$-stable open subscheme of $V$ of the form
\[
U=V-\bigcup_{L\in\cL} L,
\]
where $\cL$ is a finite $G$-stable collection of proper linear subspaces; each individual $L$ need not be $G$-stable. The new stratum will be formed by the locus in $U$ with the most singular $G$-action, namely the locus with stabilizer subgroups of maximal order. 
\subsection{The inductive strata}
To be more precise, we introduce the following notation. For a subgroup $H\le G$, denote
\begin{align*}
    U^H&=\text{$H$-fixed-point subscheme of $U$},\\
    U_H&=U^H-\bigcup_{g\in G\setminus H} U^{\langle H,g\rangle},
\end{align*}
where $\langle H,g\rangle$ is the subgroup of $G$ generated by $H$ and $g$. For $U\ne\varnothing$, set
\[
m=\max\{|H|: H\le G,\  U_H\neq \varnothing\}.
\]
The new stratum of $U$ will be
\[
S=\bigcup_{|H|=m} U_H.
\]
By $U_H\neq \varnothing$ we mean that $U_H$ is non-empty as a scheme, or equivalently $U_H$ admits a $K$-point for some field extension $K/k$. A $K$-point of $U_H$ is a $K$-point of $U$ whose stabilizer subgroup is \textit{exactly} $H$. Therefore $U_H\cap U_{H'}=\varnothing$ for $H\ne H'$ and $S$ is in fact a \textit{disjoint} union.

Now, under the above notation, we study the properties of $S$ and $S/G$.
\begin{lemma} For a subgroup $H\le G$ of order $|H|=m$, we have $U_H=U^H$. As a consequence, each $U_H$, and therefore $S$, is smooth and closed in $U$.
\end{lemma}
\begin{proof}
    The first assertion follows from the maximality of $m$. The second assertion follows from $U_H=U^H=V^H\cap U$. Indeed, since $U$ is open in $V$, $U_H$ is open in the linear space $V^H$ and therefore smooth over $k$. Meanwhile, $V^H$ is closed in $V$ and so $U_H$ is closed in $U$.
\end{proof}
\begin{lemma}
    Let $\Conj_m(G)$ be the set of conjugacy classes of subgroups of order $m$ of $G$. Then
    \begin{align*}
         S&=\coprod_{[H]\in \Conj_m(G)} G\cdot U_H,\\
         S/G&\simeq \coprod_{[H]\in \Conj_m(G)} U_H/N_G(H).
    \end{align*}
    Here $N_G(H)$ is the normalizer of $H$ in $G$. The quotients $U_H/N_G(H)$ and $S/G$ are smooth closed subschemes of $U/G$.
\end{lemma}
\begin{proof}
    The first identity follows from $g\cdot U_H= U_{gHg^{-1}}$ for $g\in G$ and thus $G\cdot U_H=\coprod_{H'\in[H]} U_{H'}$. Now notice that $N_G(H)$ preserves $U_H$ so that $G\cdot U_H=G\times^{N_G(H)} U_H$. Taking quotients gives $(G\cdot U_H)/G\cong U_H/N_G(H)$. Moreover, $H$ acts trivially on $U_H$ and the quotient group $N_G(H)/H$ acts freely on $U_H$ by the definition of $U_H$. So $U_H\to U_H/N_G(H)$ is finite \'etale and $U_H/N_G(H)$ is smooth. The closed immersion into $U/G$ follows because taking $G$-invariants is exact when $|G|$ is invertible in $k$.
\end{proof}

\subsection{Nisnevich tubular neighborhoods}
Next, we move on to describe the normal neighborhood of $U_H$ and $U_H/N_G(H)$ in $U$ and $U/G$ respectively. 

Since $|G|$ is invertible in $k$, so are the orders of its subgroups. Thus the map
\[
e_H=\frac{1}{|H|}\sum_{h\in H}h: V\to V
\]
is well-defined, $k$-linear, and idempotent: $e_H^2=e_H$. Put $M_H=\ker e_H$. Then
\[
V=V^H\oplus M_H
\]
is an $N_G(H)$-equivariant splitting. This describes the normal bundle of $V^H$ in $V$ and consequently describes the normal bundle of $U_H=U\cap V^H$ in $U$. More precisely, let $E_H=U_H\times M_H$ be the trivial bundle over $U_H$ with fiber $M_H$, and let $\phi_H: E_H\to V$ be given by $\phi_H(x,v)=x+v$. Then $\phi_H$ maps $\phi_H^{-1}(U)$ isomorphically to a Zariski open neighborhood of $U_H$ in $U$, and maps the zero section of $E_H$ isomorphically to $U_H$.

In order to pass this on to a Nisnevich neighborhood of $U_H/N_G(H)\subset U/G$, we need the following lemma.

\begin{lemma}\label{lem:invariants-completion}
    Let a finite group $G$ act on a ring $A$, and let $\mathfrak{a}\subset A$ be an ideal with stabilizer subgroup $N=\{g\in G: g\mathfrak{a}=\mathfrak{a}\}$. Assume that $B=A^G$ is Noetherian and $A$ is a finite $B$-module, as holds, for example, when $A$ is a finitely generated $k$-algebra. Suppose that the ideals $g\mathfrak{a}$, indexed by $gN\in G/N$, are pairwise comaximal, i.e.\ if $gN\neq hN$ then $g\mathfrak{a}+h\mathfrak{a}=A$. Denote
    \[
   \quad J=\bigcap_{gN\in G/N} g\mathfrak{a}, \quad I=J\cap B.
    \]
    Then the natural map $\widehat{B}_I\to (\widehat{A}_\mathfrak{a})^N$ is an isomorphism.
\end{lemma}
\begin{proof}
    We prove the following sequence of natural isomorphisms:
    \[
    \widehat{B}_I\simeq (\widehat{A}_J)^G\simeq (\prod_{gN\in G/N}\widehat{A}_{g\mathfrak{a}})^G\simeq (\widehat{A}_\mathfrak{a})^N.
    \]
    The third isomorphism is given by the map 
    \[(\prod_{gN\in G/N}\widehat{A}_{g\mathfrak{a}})^G\to (\widehat{A}_\mathfrak{a})^N,\quad (x_{gN})_{gN\in G/N}\mapsto x_N,\]
    whose inverse is $x\mapsto (g\cdot x)_{gN\in G/N}$. The second isomorphism follows from the Chinese remainder theorem applied to comaximal ideals $(g\mathfrak{a})^n$.
    
    For the first isomorphism, we claim the equality of $A$-ideals $\sqrt{IA}=\sqrt{J}$. Indeed, it is clear that $IA\subset J$ because $I\subset J$ and $J$ is an ideal in $A$. On the other hand, for $x\in J$, consider the polynomial with $G$-invariant coefficients:
    \[
    P_x(T)=\prod_{g\in G}(T-g\cdot x)=T^{|G|}+c_1 T^{|G|-1}+\dots+c_{|G|}.
    \]
    Since $J$ is $G$-stable, all the non-leading coefficients belong to $J\cap A^G=I$. Hence $P_x(x)=0$ implies that $x^{|G|}\in IA$. This proves that $J\subset \sqrt{IA}$ and consequently $\sqrt{J}=\sqrt{IA}$.
    
 Therefore the $J$- and $IA$-adic topologies on $A$ are equivalent, implying that $\widehat{A}_J\simeq \widehat{A}_{IA}$. Since $A$ is a finite module over the Noetherian ring $B$, $\widehat{A}_{IA}$ is isomorphic to $A\otimes_B \widehat{B}_I$. Finally, since taking invariants is an equalizer and $\hat{B}_I$ is flat over the Noetherian ring $B$, one has $(A\otimes_B \widehat{B}_I)^G\simeq A^G\otimes_B \widehat{B}_I=\widehat{B}_I$. Combining these isomorphisms gives $(\widehat{A}_J)^G\simeq \widehat{B}_I$ as desired.
\end{proof}

In geometric terms, we have the following.
\begin{corollary}\label{cor:geometric-completion}
Let $q\colon X\to X/G$ be a finite scheme-theoretic quotient of a finite-type $k$-scheme by a finite group $G$, and let $Z\subset X$ be a closed subscheme with stabilizer subgroup $N=\{g\in G: gZ=Z\}$ of order $|N|$ invertible in $k$. Suppose that the distinct translates $gZ$ for $gN\in G/N$ are pairwise disjoint, i.e. $gZ\times_X hZ=\varnothing$ for $gN\neq hN$.

Then the natural morphism $Z/N\to X/G$ is a closed immersion, and there is a canonical isomorphism of formal schemes
\[
\widehat{X/G}_{\,Z/N}
\simeq
(\widehat{X}_{\,Z})/N.
\]
Here the formal quotient on the right is affine locally defined by taking $N$-invariants of the completed coordinate ring.
\end{corollary}
\begin{proof}
The assertion is local on $Y=X/G$. Let $Q=\Spec(B)\subset Y$ be an affine open subscheme. Since $q$ is finite,
\[
q^{-1}(Q)=\Spec(A)
\]
is a $G$-stable affine open subscheme of $X$, and $B=A^G$. If $Z\times_{X} q^{-1}(Q)=\varnothing$ then there is nothing to prove; we therefore assume it is non-empty. Let $\mathfrak a\subset A$ be the ideal defining $Z\times_{X} q^{-1}(Q)$, and put
\[
J=\bigcap_{gN\in G/N}g\mathfrak a,
\quad
I=J\cap B.
\]
The disjointness assumption implies that the distinct translates $g\mathfrak a$ are pairwise comaximal. Therefore, the Chinese remainder theorem gives
\[
A/J\simeq \prod_{gN\in G/N}A/g\mathfrak{a}.
\]
Consider the composition
\[
\rho: B=A^G\hookrightarrow A\twoheadrightarrow A/J\twoheadrightarrow A/\mathfrak{a},
\]
in which the last map is the projection onto the factor indexed by $N\in G/N$.

We claim that $\rho$ maps $B$ surjectively onto $(A/\mathfrak{a})^N$ with kernel $I$. It then follows that $\Spec(B/I)$ is naturally identified with $(Z/N)\times_{X/G} Q$, proving that $Z/N\to X/G$ is a closed immersion. Furthermore, \Cref{lem:invariants-completion} gives $\widehat B_I\simeq(\widehat A_{\mathfrak a})^N$. These isomorphisms are natural under restriction to smaller affine open subschemes of $Y$, and therefore glue to the asserted isomorphism of formal completions.

It remains to prove the claim. To see that $\ker \rho=I$, it is clear that $\ker \rho=B\cap \mathfrak{a}$. If $x\in B\cap\mathfrak{a}$ then $x=g\cdot x\in g\mathfrak{a}$. Hence $x\in B\cap J=I$, showing that $B\cap \mathfrak{a}=I$. To see that $\rho$ takes $B$ onto $(A/\mathfrak{a})^N$, it is clear that the image of $\rho$ is contained in $(A/\mathfrak{a})^N$. Meanwhile, the Chinese remainder theorem gives, for each $\bar{x}\in (A/\mathfrak{a})^N$, an element $x\in A$ such that
\[
x\equiv \overline{x}\pmod {\mathfrak{a}}, \quad x\equiv 0\pmod {g\mathfrak{a}}, \ (gN\neq N).
\]
Notice that $(\sum_{n\in N} n\cdot x)/|N|$ is $N$-invariant and satisfies the equations above, so we may assume that $x$ is $N$-invariant.  Then $\sum_{gN\in G/N} g\cdot x$ is well-defined, $G$-invariant and is mapped onto $\bar{x}$ by $\rho$. The proof is now complete.
\end{proof}

Let us apply \Cref{cor:geometric-completion} to analyze the neighborhood of $U_H/N_G(H)$ in $U/G$.
\begin{lemma}[Nisnevich neighborhood]\label{lem:nisnevich-neighborhood}
    There exists a Zariski open neighborhood $T_H$ of $(U_H\times\{0\})/N_G(H)$ in $E_H/N_G(H)$ with an \'etale morphism $T_H\to U/G$ such that the inverse image of $U_H/N_G(H)$ is exactly $(U_H\times\{0\})/N_G(H)$.
\end{lemma}
\begin{proof}
    If $U_H$ is empty, take $T_H=\varnothing$. We may therefore assume $U_H\ne\varnothing$. Recall that $\phi_H: E_H\to V$ takes $W:=\phi_H^{-1}(U)$ isomorphically onto an open neighborhood of $U_H$ in $U$ and maps $Z:=U_H\times\{0\}$ isomorphically onto $U_H$. In particular it identifies the completion of $W$ along $Z$ with the completion of $U$ along $U_H$. Write $N=N_G(H)$ for simplicity.

 Applying \Cref{cor:geometric-completion} to the $G$-action on $U$ and the closed subscheme $U_H$, we have
 \[
 \widehat{U/G}_{\, (U_H/N)}\simeq (\widehat{U}_{\,U_H})/N.
 \]
Similarly, applying \Cref{cor:geometric-completion} to the $N_G(H)$-action on $W$ and the closed subscheme $Z$, we have that
\[
\widehat{W/N}_{\, Z/N}\simeq (\widehat{W}_{\,Z})/N.
\]
It follows that the map
    \[
    f_H: W/N\to U/G
    \]
    induced by $\phi_H|_W$ induces an isomorphism of formal completions
\[
\widehat{W/N}_{\,Z/N}\simeq\widehat{U/G}_{\,U_H/N}.
\]
At corresponding points of $Z/N$ and $U_H/N$, this gives isomorphic completed local rings and residue fields. Since $f_H$ is of finite type between Noetherian schemes, the criterion for \'etaleness in terms of completed local rings shows that $f_H$ is \'etale along $Z/N$; see \cite[Proposition 17.6.3]{grothendieck1967egaiv4}. It is therefore \'etale on an open neighborhood $T_H'\supset Z/N$.
     
     Finally we shrink $T_H'$ to obtain the desired neighborhood as follows. Consider $F_H:=T_H'\times_{U/G}(U_H/N)$. The morphism $F_H\to U_H/N$ is \'etale, and the isomorphism $f_H|_{Z/N}: Z/N\xrightarrow{\sim} U_H/N$ defines a section whose image is $Z/N$. Since a section of an \'etale morphism is an open immersion, $Z/N$ is open in $F_H$. Moreover, $F_H$ is closed in $T'_H$. Hence $F_H\setminus(Z/N)$ is closed in $T'_H$. Define
\[
T_H=T'_H\setminus(F_H\setminus(Z/N)).
\]
Then $f_H|_{T_H}$ is \'etale and $T_H\times_{U/G}(U_H/N)=Z/N$.
    \end{proof}

\subsection{Good stratifications}
    
Now we are ready to state and prove the main theorem of this section. A \textit{Nisnevich distinguished square} is a pullback diagram
\begin{equation}
		\begin{tikzcd}
			p^{-1}(U)\ar[r]\ar[d] & V\ar[d,"p"]\\
			U\ar[r,"j"] & X
		\end{tikzcd}
\end{equation}
in which $p$ is an \'etale morphism, $j$ is an open immersion and $p^{-1}(X-U)\to X-U$ is an isomorphism. If $p$ is an open immersion, then the last condition becomes the requirement $U\cup V=X$. In this case, such a square is called a \textit{Zariski distinguished square}.
\begin{definition}[good stratification]\label{def:good-stratification}
	Let $X$ be a $k$-scheme. We define recursively what it means for $X$ to admit a \textit{good stratification of length at most $r$}. Length zero means that $X$ is smooth or empty. For $r>0$, it means either that $X$ admits a good stratification of length at most $r-1$, or that the following data are given:
	\begin{enumerate}
		\item a Nisnevich distinguished square of $k$-schemes
		\begin{equation}\label{eq:good-stratification-nisnevich-square}
		\begin{tikzcd}
			T^\circ\ar[r,"i"]\ar[d] & T\ar[d]\\
			X^\circ\ar[r] & X
		\end{tikzcd}
		\end{equation}
		in which $X^\circ$ admits a good stratification of length at most $r-1$;
		\item a Zariski distinguished square of $k$-schemes
		\begin{equation}\label{eq:good-stratification-zariski-square}
		\begin{tikzcd}
			T^\circ\ar[r,"i"]\ar[d] & T\ar[d]\\
			E^\circ\ar[r] & E
		\end{tikzcd}
		\end{equation}
		in which $E^\circ$ admits a good stratification of length at most $r-1$, $Y=E-E^\circ$ is smooth, and there is an $\aff^1$-deformation retraction of $E$ onto $Y$.
	\end{enumerate}
\end{definition}

\begin{theorem}[existence of good stratification]\label{thm:good-stratification-existence}
    Let $G$ be a finite group of order $|G|$ invertible in $k$, and let $V$ be a finite-dimensional $k$-vector space with a linear $G$-action. Let
    \[
    U=V-\bigcup_{L\in \cL} L
    \]
    be the complement of a finite $G$-stable collection of proper linear subspaces. Then $U/G$ admits a good stratification.
\end{theorem}
\begin{proof} We argue simultaneously for all finite groups by induction on the maximal geometric stabilizer order $m$. If $U$ is empty there is nothing to prove. If $m=1$, the action is free, so $U\to U/G$ is finite \'etale and $U/G$ is smooth. Assume $m>1$ and follow the notation in the preceding discussion. Taking the coproduct of the maps constructed in \Cref{lem:nisnevich-neighborhood} over $\Conj_m(G)$ yields an \'etale morphism
\[
f: T'=\coprod_H T_H\to U/G,
\]
whose restriction to $Y=\coprod_H \left((U_H\times\{0\})/N_G(H)\right)$ is an isomorphism $Y\xrightarrow{\simeq} S/G$. The last part of the proof of \Cref{lem:nisnevich-neighborhood} shows that one can further shrink $T'$ to obtain a neighborhood $T$ of $Y$ such that the preimage of $S/G$ in $T$ is exactly $Y$. Then taking
\[
T^\circ=T-Y, \quad X=U/G, \quad X^\circ=(U-S)/G,
\]
gives the required square in \Cref{eq:good-stratification-nisnevich-square}.
Meanwhile, taking
\[
E=\coprod_H E_H/N_G(H), \quad E^\circ= E-Y,
\]
gives the required square in \Cref{eq:good-stratification-zariski-square}.
Indeed, fiberwise scalar multiplication on $E_H=U_H\times M_H$ defines a morphism
\[
E_H\times\aff^1\to E_H, \, \left((x,v),t\right)\mapsto (x,tv).
\]
This morphism is $N_G(H)$-equivariant and descends to an $\aff^1$-deformation retraction of $E_H/N_G(H)$ onto $(U_H\times\{0\})/N_G(H)$. Taking the coproduct gives an $\aff^1$-deformation retraction of $E$ onto $Y$.

It remains to prove that $X^\circ$ and $E^\circ$ admit good stratifications. Notice that
\begin{align*}
	U^\circ&:=U-S= V-\bigcup_{L\in\cL} L-\bigcup_{|H|=m} V^H,\\
	E_H^\circ&:= E_H-(U_H\times\{0\})=V-\bigcup_{L\in\cL}\left((V^H\cap L)\oplus M_H\right)-V^H,
\end{align*}
When nonempty, these are complements of finite collections of proper linear subspaces stable under $G$ and $N_G(H)$, respectively. For $U^\circ$, all stabilizers have order less than $m$ by construction. At a geometric point $(x,v)$ of $E_H^\circ$, the stabilizer in $N_G(H)$ is contained in $\Stab_{N_G(H)}(x)=H$. Equality would imply $v\in M_H^H=0$, contradicting $v\ne0$. Thus these stabilizers also have order less than $m$. The induction hypothesis gives good stratifications of $X^\circ=U^\circ/G$ and each $E_H^\circ/N_G(H)$. Finite disjoint unions preserve good stratifications, by taking disjoint unions of their defining squares after enlarging their length bounds to a common bound. Hence $E^\circ$ also admits a good stratification, completing the induction.
\end{proof}

\begin{example}
Assume that $\operatorname{char}k\ne2,3$, and let $\symgrp_3$ act by permuting the coordinates on
\[
V=\{(x_1,x_2,x_3)\in\aff^3:x_1+x_2+x_3=0\}\simeq\aff^2.
\]
Writing $a=x_1x_2+x_1x_3+x_2x_3$ and $b=x_1x_2x_3$, we have $X=V/\symgrp_3\simeq\aff^2_{a,b}$. Put $D=-4a^3-27b^2$. The good stratification from the theorem is
\[
S_0=\{0\},\quad S_1=\{D=0\}-\{0\},\quad S_2=\{D\ne0\},
\]
with stabilizers $\symgrp_3$, conjugates of $\symgrp_2$, and the trivial group, respectively. The three reflection lines in $V$ map to the cuspidal discriminant curve in $X$, as shown in \Cref{fig:symmetric-group-stratification}. Although $X$ happens to be smooth, this finer stratification illustrates the construction above.

\begin{figure}[htbp]
\centering
\begin{tikzpicture}[x=0.85cm,y=0.85cm,>=stealth,font=\small]
  \fill[black!3] (-2.15,-1.6) rectangle (2.15,1.6);
  \draw[black!25] (-2.15,-1.6) rectangle (2.15,1.6);
  \draw[blue!65!black,thick] (-1.95,0)--(1.95,0);
  \draw[blue!65!black,thick] (-1.02,-1.4)--(1.02,1.4);
  \draw[blue!65!black,thick] (-1.02,1.4)--(1.02,-1.4);
  \fill[red!75!black] (0,0) circle (2pt);
  \node[above] at (0,1.82) {$V\simeq\aff^2$};
  \node[below,align=center] at (0,-1.78) {three reflection lines\\$x_i=x_j$};
  \draw[->,thick] (2.4,0)--(4.15,0) node[midway,above] {$V\to V/\symgrp_3$};
  \fill[black!3] (4.45,-1.6) rectangle (8.75,1.6);
  \draw[black!25] (4.45,-1.6) rectangle (8.75,1.6);
  \draw[->,black!35] (4.65,0)--(8.52,0) node[below] {$a$};
  \draw[->,black!35] (7.78,-1.42)--(7.78,1.42) node[right] {$b$};
  \draw[blue!65!black,thick,domain=-1.08:1.08,samples=81,variable=\t]
    plot ({7.78-2.35*\t*\t},{1.04*\t*\t*\t});
  \fill[red!75!black] (7.78,0) circle (2pt);
  \node[above] at (6.6,1.82) {$X\simeq\aff^2$};
  \node[blue!65!black] at (5.25,0.83) {$S_1$};
  \node[red!75!black,above right] at (7.78,0.02) {$S_0$};
  \node at (6.67,1.13) {$S_2$};
  \node[below,align=center] at (6.6,-1.78) {discriminant curve\\$4a^3+27b^2=0$};
\end{tikzpicture}
\caption{The good stratification on $\mathbf{A}^2/\symgrp_3$.}
\label{fig:symmetric-group-stratification}
\end{figure}

At $S_0$, take $T=E=X$, with contraction $(a,b,t)\mapsto(t^2a,t^3b)$. For $S_1\subset X-S_0$, set
\[
E=\mathbf G_{m,u}\times\aff^1_w,\quad Y=\{w=0\},\quad
T=\{9u^2-w\ne0\}\subset E.
\]
The morphism
\[
f:T\longrightarrow X-S_0,\quad (u,w)\longmapsto(-3u^2-w,-2u^3+2uw)
\]
is \'etale and identifies $f^{-1}(S_1)=Y$ with $S_1$. Taking $X^\circ=S_2$, $T^\circ=T-Y$, and $E^\circ=E-Y$ therefore gives the two distinguished squares of \Cref{def:good-stratification}. Both $X^\circ$ and $E^\circ$ are smooth, and $(u,w,t)\mapsto(u,t^2w)$ contracts $E$ onto $Y$.
\end{example}

\section{Voevodsky's change-of-category comparison}\label{sec:change-of-category}
Throughout this section, $k$ is a perfect field and all schemes are separated and of finite type over $k$. The category of separated $k$-schemes of finite type will be denoted by $\sch/k$. Let $\sm/k$, $\smqp/k$ and $\snqp/k$ respectively be the full subcategories of $\sch/k$ of smooth schemes, smooth quasi-projective schemes and seminormal quasi-projective schemes over $k$. There are canonical embeddings
\[
j: \smqp/k\hookrightarrow \sm/k,\quad i: \smqp/k\hookrightarrow\snqp/k.
\]
The embedding $j$ is well-behaved in the sense that it induces an equivalence between the categories of sheaves in any topology which is at least as strong as the Zariski one. As a consequence, the motivic homotopy categories $H_{\nis,\aff^1}((\sm/k)_+)$ and $H_{\nis,\aff^1}((\smqp/k)_+)$ (see \Cref{subsec:comparison-preliminaries} for the definition) are equivalent.

A key step in Voevodsky's work \cite{voevodsky2010motivic} on motivic Eilenberg--MacLane spaces is to show \cite[Theorem 1.21]{voevodsky2010motivic} that $i$ is also well-behaved: the right adjoints defined by $i$ commute with reduced motive functors on homotopy categories. This then allows him to transfer information obtained in $\snqp/k$ to $\smqp/k$ and then to $\sm/k$. The proof of \cite[Theorem 1.21]{voevodsky2010motivic} is the crucial place where resolution of singularities is used.

The goal of this section is to prove, without using resolution of singularities, a weaker version (\Cref{thm:transfer-comparison}) of \cite[Theorem 1.21]{voevodsky2010motivic} that is sufficient for our applications.

\subsection{Preliminaries}\label{subsec:comparison-preliminaries}
As we will frequently cite results from \cite{voevodsky2010motivic}, we decide to closely follow its notation so that the reader can easily compare.

Let $C$ be an admissible category, i.e. a full subcategory of $\sch/k$ that contains $\Spec(k), \aff^1$ and the domain of every \'etale morphism with target in $C$, and such that $C$ is closed under finite products and finite coproducts. In particular, $C$ contains every affine space, every open subscheme of one of its objects, and every connected component of one of its objects.

Denote by $C_+$ the category of schemes $X_+=X\amalg\Spec(k)$ with $X\in C$ and basepoint-preserving morphisms, and by $\cor(C,R)$ the category of finite correspondences over $C$ with coefficients in a commutative unital ring $R$.

The functor $[-]_R: C\to \cor(C,R)$ that is the identity on objects and takes morphisms to their graphs factors through the functor $(-)_+: C\to C_+$ by adjoining a disjoint basepoint. We denote the corresponding functor $C_+\to \cor(C,R)$ by $\Lambda_R$; in particular, $\Lambda_R(X_+)=[X]_R$. This functor commutes with finite coproducts and therefore defines a pair of adjoint functors:
\[
\Lambda_R^l: \Rad(C_+)\leftrightarrows \Rad(\cor(C,R)): \Lambda_R^r
\]
on the associated categories of radditive presheaves of sets.

Since $\cor(C,R)$ is an $R$-linear additive category, a radditive presheaf of sets may be identified with a presheaf of $R$-modules; therefore the category $\Rad(\cor(C,R))$ is the category of \textit{presheaves with transfers} and $\Lambda^r_R$ forgets the transfers.

Let $i: C\to D$ be an embedding of admissible categories, e.g. $i: \smqp/k\hookrightarrow\snqp/k$. The commutative square
\[
\begin{tikzcd}
	C_+\ar[r,"i_+"]\ar[d,"\Lambda_R"] & D_+\ar[d,"\Lambda_R"] \\
	\cor(C,R)\ar[r,"i_{\tr}"] & \cor(D,R)
\end{tikzcd}
\]
yields the following square
\begin{equation}
	\begin{tikzcd}
	H_{\nis,\aff^1}(D_+)\ar[r,"i_{+,\rad}"]\ar[d,"\LL\Lambda^l_R"] & H_{\nis,\aff^1}(C_+)\ar[d,"\LL\Lambda^l_R"] \\
	H_{\nis,\aff^1}(\cor(D,R))\ar[r,"i_{\tr,\rad}"] & H_{\nis,\aff^1}(\cor(C,R))
\end{tikzcd}
\end{equation}
equipped with a natural transformation
\begin{equation}
	\theta: \LL\Lambda^l_R i_{+,\rad}\to i_{\tr,\rad}\LL\Lambda^l_R.
\end{equation}
Here $H_{\nis,\aff^1}(-)$ means the corresponding homotopy category constructed by localizing the category of simplicial radditive presheaves $\Delta^{\op}\Rad(-)$ first with respect to projective equivalences (i.e.\ pointwise simplicial weak equivalences) and then with respect to Nisnevich--$\aff^1$ equivalences. The notation $\LL(-)$ means the corresponding total left derived functor.

\begin{definition}[$\theta$-compatible]
	We say $Y\in H_{\nis,\aff^1}(D_+)$ is $\theta$-compatible if $\theta_Y$ is an isomorphism. For $X\in D$, we say $X$ is $\theta$-compatible if $\theta_{X_+}$ is an isomorphism.
\end{definition}

The rest of this section is devoted to proving that $U/G$ as in \Cref{thm:good-stratification-existence} is $\theta$-compatible. For this, we need to establish some useful properties of $\theta$.

\subsection{Properties of $\theta$}
An object $Y$ of $H_{\nis,\aff^1}(\cC_+)$, for an admissible category $\cC$, is a pointed \textit{motivic $\cC$-space}, and $\LL\Lambda^l_{R}Y\in H_{\nis,\aff^1}(\cor(\cC,R))$ is called its \textit{reduced $R$-motive}. Thus $\LL\Lambda^l i_{+,\rad}$ takes the reduced motive after restriction, whereas $i_{\tr,\rad}\LL\Lambda^l$ restricts the reduced motive. The natural transformation $\theta:\LL\Lambda^l i_{+,\rad}\to i_{\tr,\rad}\LL\Lambda^l$ compares these constructions.

\begin{lemma}\label{lem:comparison-hocolimits}
	The functors $\LL\Lambda^l i_{+,\rad}$ and $i_{\tr,\rad}\LL\Lambda^l$ preserve small homotopy colimits, compatibly with $\theta$.
\end{lemma}
\begin{proof}
By \cite[Theorem~1.7]{voevodsky2010motivic}, the functors $\LL\Lambda^l_R$, for $\cC=C,D$, are derived left adjoints. They therefore preserve small homotopy colimits by the general result for left Quillen functors \cite[Theorem~19.4.5(1)]{hirschhorn2003model}. It remains to show the restrictions $i_{+,\rad}$ and $i_{\tr,\rad}$ also preserve small homotopy colimits.

In the projective homotopy categories $H(-)$ of simplicial radditive presheaves $\Delta^{\op}\Rad(-)$, small homotopy colimits are computed using coproducts and bisimplicial diagonals \cite[Definition~18.1.2 and Theorems~15.11.6,~15.11.11]{hirschhorn2003model}. Both restrictions commute with these operations and hence preserve these homotopy colimits.

Meanwhile, both restrictions preserve Nisnevich--$\aff^1$ equivalences \cite[Corollary~1.20 and Theorem~A.16]{voevodsky2010motivic}, so their induced functors on $H_{\nis,\aff^1}(-)$ commute with localization. The localization functors preserve homotopy colimits as derived left Quillen functors \cite[Appendix~A.3, pp.~92--93]{voevodsky2010motivic}, so both restrictions also preserve small homotopy colimits in $H_{\nis,\aff^1}(-)$. This proves the assertion for $\LL\Lambda^l i_{+,\rad}$ and $i_{\tr,\rad}\LL\Lambda^l$; compatibility with $\theta$ follows from naturality.
\end{proof}

\begin{lemma}\label{lem:comparison-properties}
	Let $k$ be a perfect field and $C\subset\sm/k$. The following hold.
	\begin{enumerate}
		\item If $X\in D$ is smooth, then $X$ is $\theta$-compatible.
		\item (two-out-of-three) Let $X\to Y\to Z$ be a homotopy cofiber sequence in $H_{\nis,\aff^1}(D_+)$. If two of $X,Y,Z$ are $\theta$-compatible, then so is the third.
		\item ($\aff^1$-invariance) If $Y\to X$ is an $\aff^1$-homotopy equivalence in $D$, then $X$ is $\theta$-compatible if and only if $Y$ is.
		\item (Nisnevich excision) Suppose we have a Nisnevich distinguished square
		\[
		\begin{tikzcd}
			Y^\circ\ar[r,"i"]\ar[d] & Y\ar[d,"p"]\\
			X^\circ\ar[r] & X
		\end{tikzcd}
		\]
		in $D$. Then $X/X^\circ$ is $\theta$-compatible if and only if $Y/Y^\circ$ is.
	\end{enumerate}
\end{lemma}
\begin{proof}
For (1), first let $U\in C$. Restricting $U_+$ from $D$ to $C$ gives the pointed presheaf represented by the same scheme $U$. Applying $\LL\Lambda_R^l$ then gives the presheaf with transfers $V\mapsto\cor(C,R)(V,U)$, for $V\in C$; the same presheaf is obtained by first applying $\LL\Lambda_R^l$ over $D$ and then restricting to $C$. Under these identifications, $\theta_{U_+}$ is the identity, so $U$ is $\theta$-compatible.

For an arbitrary smooth $X\in D$, choose a finite Zariski open cover of $X$ whose members admit \'etale maps to affine spaces; these opens and their finite intersections belong to $C$ by admissibility. Let $U_\bullet$ be the \v{C}ech nerve of this cover: each $U_n$ is a finite disjoint union of intersections of the covering opens, so $U_n\in C$. The augmentation $U_\bullet\to X$ is a Zariski, hence Nisnevich, local equivalence \cite[Section~2.1, Lemma~1.15]{morel1999homotopy}, giving
\[
X_+\simeq\hocolim\nolimits_{[n]\in\Delta^{\op}}(U_n)_+
\quad\text{in }H_{\nis,\aff^1}(D_+).
\]
Each $U_n$ is $\theta$-compatible, so \Cref{lem:comparison-hocolimits} shows that $X$ is also $\theta$-compatible.

For (2), by \Cref{lem:comparison-hocolimits}, applying $\LL\Lambda^l i_{+,\rad}$ and $i_{\tr,\rad}\LL\Lambda^l$ to the given cofiber sequence, we obtain two cofiber sequences in $H_{\nis,\aff^1}(\cor(C,R))$. The natural transformation $\theta$ gives a morphism between them. By \cite[Theorem 1.15 and Proposition 1.17]{voevodsky2010motivic}, there exists a fully faithful embedding of $H_{\nis,\aff^1}(\cor(C,R))$ into a triangulated category, under which cofiber sequences become distinguished triangles. Two-out-of-three for morphisms of distinguished triangles in a triangulated category therefore proves (2). Finally, (3) and (4) follow from naturality of $\theta$, since $Y_+\to X_+$ in (3) and the Nisnevich excision map $Y/Y^\circ\to X/X^\circ$ in (4) are isomorphisms in $H_{\nis,\aff^1}(D_+)$.
\end{proof}

In most applications below, we apply two-out-of-three to the cofiber sequence
\[
U_+\to X_+\to X/U,
\]
in which $X\in D$ and $U\hookrightarrow X$ is an open subscheme.

\subsection{Good stratification implies $\theta$-compatibility}
\Cref{thm:transfer-comparison} is now a formal consequence of \Cref{lem:comparison-properties}.
\begin{theorem}\label{thm:transfer-comparison}
	Let $k$ be a perfect field, and $C\subset\sm/k$. Suppose that $X\in D$ admits a good stratification in $D$, namely all the data of the good stratification are contained in $D$. Then $X$ is $\theta$-compatible, i.e.
	\[
	\theta_{X_+}: \LL\Lambda^l_R i_{+,\rad}(X_+)\to i_{\tr,\rad}\LL\Lambda^l_R(X_+)
	\]
	is an isomorphism.
\end{theorem}
\begin{proof}
	Apply induction on the length of the good stratification. In the base case $X$ is smooth and therefore $\theta$-compatible. In general, following the notation of \Cref{def:good-stratification}, $X^\circ$ and $E^\circ$ are $\theta$-compatible by the induction hypothesis. Since $E-E^\circ$ is smooth and $E$ is $\aff^1$-equivalent to $E-E^\circ$, $E$ is $\theta$-compatible. Then by two-out-of-three $E/E^\circ$ is $\theta$-compatible. By Nisnevich excision $T/T^\circ$ and then $X/X^\circ$ are $\theta$-compatible. Finally, again by two-out-of-three, $X$ is $\theta$-compatible.
\end{proof}

Recall that an admissible category is \textit{$f$-admissible} if it is in addition closed under the formation of quotients with respect to actions of finite groups. For example, $\snqp/k$ and $\qp/k$ (quasi-projective $k$-schemes) are $f$-admissible.
\begin{corollary}\label{cor:comparison-finite-linear-quotients} Let $k$ be a perfect field and $i\colon C\hookrightarrow D$ an inclusion of admissible categories with $C\subset\sm/k$ and $D$ $f$-admissible. Under the assumptions of \Cref{thm:good-stratification-existence}, $U/G$ is $\theta$-compatible for every coefficient ring $R$.
\end{corollary}
\begin{proof}
Admissibility puts every affine space in $D$, hence also the vector spaces $V$, $V^H$, and $M_H$. The schemes $U$ and $U_H=U\cap V^H$ are open subschemes of $V$ and $V^H$, respectively, so they lie in $D$, as does $E_H=U_H\times M_H$. All remaining steps in the construction of \Cref{thm:good-stratification-existence}, including its recursive stages, use open subschemes, finite products, finite disjoint unions, and finite group quotients. The $f$-admissibility of $D$ therefore keeps every auxiliary scheme in $D$. Thus $U/G$ admits a good stratification in $D$, and \Cref{thm:transfer-comparison} applies.
\end{proof}

\section{Motivic Eilenberg--MacLane spaces}\label{sec:motivic-em}
Throughout this section, let $k$ be a field of characteristic zero and let $i:C\hookrightarrow D$ be an inclusion of admissible categories with $C\subset\sm/k$ and $D$ $f$-admissible. 

The goal of this section is to prove that the motivic Eilenberg--MacLane space $K(A,2n,n)$ is $\theta$-compatible for an arbitrary abelian group $A$. Our strategy is to use the results of \Cref{sec:change-of-category} to show that all finite symmetric powers of $S^{2n,n}=\aff^n/(\aff^n-\{0\})$ are $\theta$-compatible. Then we deduce from the motivic Dold--Thom theorem that $K(\ZZ,2n,n)$ is $\theta$-compatible. For finitely generated $A$, we use a simplicial resolution by finitely generated free abelian groups, and then pass to arbitrary $A$ by filtered colimits.

\subsection{Symmetric powers}\label{subsec:symmetric-powers} Since $D$ is $f$-admissible, for $X\in D$ and $d\ge 0$, the $d$-fold symmetric power and the reduced $d$-fold symmetric power
\begin{align*}
    \Sym^d(X_+)&:=(X_+)^d/\symgrp_d,\\
    \widetilde{\Sym}^d(X_+)&:=(X_+)^{\wedge d}/\symgrp_d=(X^d/\symgrp_d)_+
\end{align*}
are well-defined in $D_+$. For example, for small values of $d$, one has
\begin{align*}
    &\Sym^0(X_+)=\Spec(k), \quad\widetilde{\Sym}^0(X_+)=\Spec(k)_+\\
    &\Sym^1(X_+)=\widetilde{\Sym}^1(X_+)=X_+.
\end{align*}
They induce derived functors
\[
\LL\Sym^d,\ \LL\widetilde{\Sym}^d: H_{\nis,\aff^1}(D_+)\to H_{\nis,\aff^1}(D_+).
\]

\begin{lemma}\label{lem:symmetric-power-cofiber}
	For every $X\in H_{\nis,\aff^1}(D_+)$ and $d\ge 1$, there is a homotopy cofiber sequence
	\[
	\LL\Sym^{d-1}(X)\to \LL\Sym^d(X)\to \LL\widetilde{\Sym}^d(X).
	\]
	The first map is induced by adjoining a basepoint.
\end{lemma}
\begin{proof}
Choose a representative $\cX\in \Delta^{\op}\Rad(D_+)$ of $X$. The resolution of \cite[Proposition A.12(1)]{voevodsky2010motivic} gives a functorial projective equivalence
\[
L_*\cX\to\cX,
\]
in which each term of $L_*\cX$ is a coproduct of representable presheaves; in particular, $L_*\cX\in\Delta^{\op}D_+^\#$ where $D_+^\#$ is the full subcategory of $\Rad(D_+)$ consisting of filtered colimits of representable presheaves. Then by \cite[Corollary 2.11]{voevodsky2010motivic}, $\Sym^d(L_*\cX)$ and $\widetilde{\Sym}^d(L_*\cX)$ respectively represent $\LL\Sym^d(X)$ and $\LL\widetilde{\Sym}^d(X)$ for $d\ge 0$.

Now \cite[Lemma 2.21]{voevodsky2010motivic} gives a termwise coprojection sequence in $\Delta^{\op} D_+^\#$
\[
\Sym^{d-1}(L_*\cX)\to\Sym^d(L_*\cX)\to\widetilde{\Sym}^d(L_*\cX).
\]
By \cite[Proposition A.18(1)]{voevodsky2010motivic}, this yields a cofiber sequence in the homotopy category of $\Delta^{\op}\Rad(D_+)$ with respect to the projective model structure. The asserted cofiber sequence follows from the fact that the Nisnevich--$\aff^1$-localization is a left Bousfield localization (see \cite[Appendix~A.3, pp.~92--93]{voevodsky2010motivic}) and therefore preserves cofiber sequences, see e.g. \cite[Proposition 6.4.1]{hovey1999model}.
\end{proof}


\begin{lemma}\label{lem:comparison-symmetric-powers} $\LL\widetilde{\Sym}^d(S^{2n,n})$ and $\LL\Sym^d(S^{2n,n})$ are $\theta$-compatible for all $n, d\ge 0$.
\end{lemma}
\begin{proof}
	For $d=0$, the two spaces are $\Spec(k)_+$ and $\Spec(k)$, respectively, and the assertion is immediate. Assume $d\ge1$. Since $S^{2n,n}=\aff^n/(\aff^n-\{0\})$ is obviously solid in the sense of \cite[Definition 2.4]{voevodsky2010motivic}, \cite[Corollary 2.11]{voevodsky2010motivic} applies to give
	\[
	\LL\widetilde{\Sym}^d(S^{2n,n})=\widetilde{\Sym}^d(S^{2n,n}).
	\]
	Then \cite[Proposition 2.12]{voevodsky2010motivic} (see also \cite[p. 4]{voevodsky2010motivic}) shows that the associated Nisnevich sheaves satisfy
	\begin{align*}
		\widetilde{\Sym}^d(S^{2n,n})&=\widetilde{\Sym}^d(\aff^n_+)/\left(\widetilde{\Sym}^d(\aff^n_+)-\widetilde{\Sym}^d(\{0\}_+)\right)\\
		&=\bigl((\aff^n)^d/\symgrp_d\bigr)\Big/\bigl(((\aff^n)^d-\{0\})/\symgrp_d\bigr).
	\end{align*}
	Now by \Cref{cor:comparison-finite-linear-quotients} both $\left((\aff^n)^d-\{0\}\right)/\symgrp_d$ and $(\aff^n)^d/\symgrp_d$ are $\theta$-compatible, and hence so is $\LL\widetilde{\Sym}^d(S^{2n,n})$ by two-out-of-three. The assertion for $\LL\Sym^d(S^{2n,n})$ follows from two-out-of-three by induction using \Cref{lem:symmetric-power-cofiber}.
\end{proof}

\subsection{Split proper Tate motives}
We now take a necessary but brief detour to study split Tate motives, before we apply the results of \Cref{subsec:symmetric-powers} to study motivic Eilenberg--MacLane spaces.

For an admissible category $\cC$ and $n\ge 0$ put
\[
	R(n)[2n]_{\cC}:=\LL\Lambda^l_{R}(S^{2n,n}) \in H_{\nis,\aff^1}(\cor(\cC,R)),
\]
and write
\[
R(n)[2n+i]_{\cC}:=\Sigma^i R(n)[2n]_{\cC}
\]
for $i\ge0$, where $\Sigma$ is the simplicial suspension. 

Following the terminology of \cite[Definition 2.60]{voevodsky2010motivic}, an object in $H_{\nis,\aff^1}(\cor(\cC,R))$ is called \textit{split proper Tate} if it is isomorphic to a direct sum of objects $R(n)[m]_{\cC}$ with $n\ge0$ and $m\ge2n$. Denote this full subcategory by $\spt_\cC$. It is closed under direct sums, simplicial suspensions, and derived tensor products, as follows from the definition and
\[
R(n)[m]\otimes^{\LL}R(n')[m']\simeq R(n+n')[m+m'],
\]
which is a consequence of
\[
S^{2n,n}\wedge S^{2n',n'}\simeq S^{2(n+n'),n+n'}
\]
by \Cref{lem:reduced-motive-smash}.

\begin{lemma}[K\"unneth formula]\label{lem:reduced-motive-smash}
For $X,Y\in H_{\nis,\aff^1}(\cC_+)$, there is a natural isomorphism
\[
\LL\Lambda_R^l(X\wedge Y)\simeq
\LL\Lambda_R^l X\otimes^\LL\LL\Lambda_R^l Y.
\]
\end{lemma}
\begin{proof}
The functor $\Lambda_R:\cC_+\to \cor(\cC,R)$ is by definition strong symmetric monoidal:
\[
\Lambda_R(U_+\wedge V_+)
\cong \Lambda_R(U_+)\otimes\Lambda_R(V_+).
\]
Its extension to simplicial radditive presheaves
\[
\Lambda_R^l:
\Delta^{\op}\Rad(\cC_+)
\longrightarrow
\Delta^{\op}\Rad(\cor(\cC,R))
\]
remains strong symmetric monoidal because $\Lambda_R^l$, $\wedge$ and $\otimes$ all commute with small colimits (separately in each variable for $\wedge$ and $\otimes$).

Moreover, by \cite[Lemma~1.3.4]{hovey1999model}, $\Lambda_R^l$ is left Quillen for the projective model structures because its right adjoint $\Lambda^r_R$ forgets transfers and hence preserves objectwise fibrations and trivial fibrations. The left derived functor $\LL\Lambda^l_R$ therefore preserves the derived monoidal products in the projective homotopy categories by \cite[Theorem~4.3.3]{hovey1999model}. We pass to motivic homotopy categories $H_{\nis, \aff^1}$ in three steps.

First, by \cite[Lemma~1.4 and its proof]{voevodsky2010motivic}, the ordinary smash product preserves projective equivalences, so it computes the derived smash product. The same lemma shows that it preserves Nisnevich--$\aff^1$ equivalences, and hence defines the smash product on $H_{\nis,\aff^1}(\cC_+)$. Second, the derived tensor product can be computed using simplicial resolutions whose terms are coproducts of representable presheaves with transfers. Tensor products preserve Nisnevich--$\aff^1$ equivalences between these resolutions by \cite[Lemma~1.6]{voevodsky2010motivic}. Hence the derived tensor product descends to $H_{\nis,\aff^1}(\cor(\cC,R))$. Finally, \cite[Theorem~1.7]{voevodsky2010motivic} shows that $\LL\Lambda_R^l$ descends from the projective homotopy categories to the motivic homotopy categories. Applying motivic localization to the projective monoidal isomorphism now gives the claimed formula, with the smash product and derived tensor product just described.
\end{proof}

\begin{lemma}\label{lem:restriction-tate-tensor}
	Let $i:C\hookrightarrow D$ be an inclusion of admissible categories. For $E_1, E_2\in\spt_D$, the canonical lax monoidal map
	\[
	\mu: i_{\tr,\rad}E_1\otimes^\LL i_{\tr,\rad}E_2\to i_{\tr,\rad}(E_1\otimes^\LL E_2)
	\]
	is an isomorphism in $H_{\nis,\aff^1}(\cor(C,R))$.
\end{lemma}
\begin{proof}
The left adjoint $\LL i_{\tr}^{\rad}$ to $i_{\tr,\rad}$ is strong symmetric monoidal. This is a formal consequence of
\[
i_{\tr}^{\rad}([U]_C\otimes[V]_C)
\cong[U\times V]_D
\cong[U]_D\otimes[V]_D
\]
for representable presheaves with transfers associated with $U,V\in C$. The proof is similar to the proof of \Cref{lem:reduced-motive-smash} and thus omitted. Moreover $\LL i_{\tr}^{\rad}$ is fully faithful by \cite[Corollary 1.20]{voevodsky2010motivic}, so its adjunction unit is an isomorphism. 

Now every split proper Tate object lies in the essential image of $\LL i_{\tr}^{\rad}$, since $\LL i_{\tr}^{\rad}(R(j)[2j]_C)\simeq R(j)[2j]_D$ and it preserves direct sums and simplicial suspensions. Write $E_1\simeq\LL i_{\tr}^{\rad}(P)$ and $E_2\simeq\LL i_{\tr}^{\rad}(Q)$. The unit isomorphisms give $i_{\tr,\rad}E_1\simeq P$ and $i_{\tr,\rad}E_2\simeq Q$. By the monoidality of $\LL i_{\tr}^{\rad}$, the map in the statement becomes the adjunction unit at $P\otimes^{\LL}Q$, hence an isomorphism.
\end{proof}

\begin{remark}
If $C\subset\sm/k$ and $k$ admits resolution of singularities, Voevodsky's K\"unneth formula \cite[Proposition~4.1.7]{voevodsky2000triangulated} implies that $i_{\tr,\rad}$ preserves derived tensor products of arbitrary objects, by passing to simplicial resolutions by coproducts of representable presheaves with transfers. \Cref{lem:restriction-tate-tensor} avoids resolution of singularities and suffices for our applications: the motivic Eilenberg--MacLane spaces to which we apply it below have split proper Tate reduced motives with field coefficients by \cite[Corollary~3.28]{voevodsky2010motivic}.
\end{remark}

\begin{corollary}\label{cor:comparison-tate-products}
Let $i:C\hookrightarrow D$ be an inclusion of admissible categories with $C\subset\sm/k$, and let $X,Y\in H_{\nis,\aff^1}(D_+)$ be $\theta$-compatible with split proper Tate reduced $R$-motives. Then $X\vee Y$, $X\wedge Y$, and $X\times Y$ are also $\theta$-compatible with split proper Tate reduced $R$-motives.
\end{corollary}
\begin{proof}
The wedge sum $X\vee Y$ is $\theta$-compatible by \Cref{lem:comparison-hocolimits}. Its reduced motive is the direct sum of those of $X$ and $Y$, since $\LL\Lambda_R^l$ is a left adjoint and hence preserves coproducts. This direct sum is split proper Tate.

For $X\wedge Y$, note that the restriction functor $i_{+,\rad}$ preserves smash products, since smash products are computed objectwise on connected schemes. Together with \Cref{lem:reduced-motive-smash}, this gives the two vertical isomorphisms in the diagram
\[
\begin{tikzcd}
\LL\Lambda_R^l i_{+,\rad}(X\wedge Y)
  \ar[r,"\theta_{X\wedge Y}"] \ar[d,"\simeq"']
& i_{\tr,\rad}\LL\Lambda_R^l(X\wedge Y) \\
\LL\Lambda_R^l i_{+,\rad}X\otimes^{\LL}\LL\Lambda_R^l i_{+,\rad}Y
  \ar[d,"\theta_X\otimes^{\LL}\theta_Y"']
& \\
i_{\tr,\rad}\LL\Lambda_R^l X\otimes^{\LL}i_{\tr,\rad}\LL\Lambda_R^l Y
  \ar[r,"\mu"']
& i_{\tr,\rad}\bigl(\LL\Lambda_R^l X\otimes^{\LL}\LL\Lambda_R^l Y\bigr)
  \ar[uu,"\simeq"']
\end{tikzcd}
\]
which commutes by the monoidal compatibility of the adjunction maps defining $\theta$. Here $\mu$ is the lax monoidal map of \Cref{lem:restriction-tate-tensor}, hence an isomorphism by the split proper Tate hypothesis, and $\theta_X\otimes^{\LL}\theta_Y$ is an isomorphism by $\theta$-compatibility. Thus $X\wedge Y$ is $\theta$-compatible, and its reduced motive is split proper Tate by \Cref{lem:reduced-motive-smash} and closure of $\spt$ under tensor products.

Finally applying two-out-of-three to $X\vee Y\to X\times Y\to X\wedge Y$ shows that $X\times Y$ is $\theta$-compatible. The left derived Quillen functor $\LL\Lambda_R^l$ preserves this homotopy cofiber sequence \cite[Proposition~6.4.1]{hovey1999model}, giving
\[
\LL\Lambda_R^l X\oplus\LL\Lambda_R^l Y
\longrightarrow\LL\Lambda_R^l(X\times Y)
\longrightarrow\LL\Lambda_R^l(X\wedge Y).
\]
The projections from $X\times Y$ onto $X$ and $Y$ induce a splitting of the first map, proving that the reduced motive of $X\times Y$ is a direct sum of those of $X\vee Y$ and $X\wedge Y$, hence split proper Tate.
\end{proof}

\subsection{Eilenberg--MacLane spaces}
In this subsection, we study motivic Eilenberg--MacLane spaces and their motives.

Following \cite[Section 3.2]{voevodsky2010motivic}, the motivic Eilenberg--MacLane space of an abelian group $A$ on an admissible category $\cC$ is given, for $n\ge0$, by
\[
K(A,2n,n)_{\cC}=\Lambda^r_{\ZZ}(A\otimes^{\LL}_{\ZZ}\ZZ(n)[2n]_{\cC})\in H_{\nis,\aff^1}(\cC_+).
\]

\begin{lemma}\label{lem:eilenberg-maclane-realization}
Let $M_\bullet\to A$ be a simplicial resolution of an abelian group $A$. Then the augmentation induces a natural isomorphism
\[
|K(M_\bullet,2n,n)_\cC|\simeq K(A,2n,n)_{\cC}
\]
in $H_{\nis,\aff^1}(\cC_+)$.
\end{lemma}
\begin{proof}
Choose a projectively cofibrant simplicial presheaf with transfers $T_\bullet$ representing $\ZZ(n)[2n]_{\cC}$. Each $T_r$ is a direct summand of a direct sum of representable presheaves with transfers, whose values are free abelian groups of finite correspondences; hence $T_r(U)$ is projective and thus flat over $\ZZ$ for every $U\in\cC$. Then $K(M_q,2n,n)_{\cC}$ is represented by $\Lambda^r_{\ZZ}(M_q\otimes_{\ZZ}T_\bullet)$, and $|K(M_\bullet)|$ is represented by the diagonal of the resulting bisimplicial presheaf by \cite[Theorem 15.11.6]{hirschhorn2003model}. For every $U$ and $r$, flatness gives a weak equivalence $M_\bullet\otimes_{\ZZ}T_r(U)\to A\otimes_{\ZZ}T_r(U)$. By \cite[Theorem 15.11.11]{hirschhorn2003model}, a map of bisimplicial sets that is a weak equivalence in each row induces a weak equivalence on realizations, hence on diagonals. Applying this to the preceding maps for each $U$ proves the assertion.
\end{proof}

\begin{theorem}\label{thm:comparison-eilenberg-maclane}
Let $k$ be a field of characteristic zero and $R$ a commutative unital ring. Suppose $i:C\hookrightarrow D$ is an inclusion of admissible categories such that $C\subset\sm/k$ and $D$ is $f$-admissible and consists of seminormal schemes. Then for every abelian group $A$ and every $n>0$, the canonical comparison morphism
\[
\theta_{K(A,2n,n)_D}:
\LL\Lambda^l_{R}K(A,2n,n)_C
\to
i_{\tr,\rad}\LL\Lambda^l_{R}K(A,2n,n)_D
\]
is an isomorphism in $H_{\nis,\aff^1}(\cor(C,R))$. Here the source is identified using the natural isomorphism $i_{+,\rad}K(A,2n,n)_D\simeq K(A,2n,n)_C$ of \cite[Lemma 3.18]{voevodsky2010motivic}.
\end{theorem}

\begin{proof}

We first treat the case $R=\FF$ with $\FF$ a field. Thanks to Voevodsky's motivic Dold--Thom package \cite[Theorem~3.7, Proposition~3.11, and Lemma~3.13]{voevodsky2010motivic} we get an equivalence
\[
K(\ZZ,2n,n)_D\simeq\LL\Sym^\infty(S^{2n,n})
\]
in $H_{\nis,\aff^1}(D_+)$. Since the quotient model $\aff^n/(\aff^n-\{0\})$ of $S^{2n,n}$ is solid, \cite[Corollaries~2.11 and~2.23]{voevodsky2010motivic} and the fact that small filtered colimits of pointed simplicial sets compute homotopy colimits \cite[XII, \S3.5]{bousfield1972homotopy}, applied objectwise, give
\[
\LL\Sym^\infty(S^{2n,n})\simeq\hocolim_{d\ge0}\LL\Sym^d(S^{2n,n}).
\]
Each finite symmetric power is $\theta$-compatible by \Cref{lem:comparison-symmetric-powers}, hence so is $K(\ZZ,2n,n)_D$ by \Cref{lem:comparison-hocolimits}. Moreover, $\LL\Lambda^l_{\FF}K(\ZZ,2n,n)_D$ is split proper Tate by \cite[Corollary 3.28]{voevodsky2010motivic}, so all finite powers of $K(\ZZ,2n,n)_D$ are $\theta$-compatible by \Cref{cor:comparison-tate-products}.

For finitely generated $A$, choose a simplicial resolution $M_\bullet\to A$ by finitely generated free abelian groups. Each $K(M_q,2n,n)_D$ is a finite Cartesian power of $K(\ZZ,2n,n)_D$, hence is $\theta$-compatible. Since geometric realization is a small homotopy colimit, $\left|K(M_\bullet,2n,n)_D\,\right|$ is $\theta$-compatible. Therefore so is $K(A,2n,n)_D$ by \Cref{lem:eilenberg-maclane-realization}.

For arbitrary $A$, write $A=\colim_\lambda A_\lambda$ as the filtered colimit of its finitely generated subgroups. Let $T_\bullet$ be as in the proof of \Cref{lem:eilenberg-maclane-realization}. Then $K(A_\lambda,2n,n)_D$ and $K(A,2n,n)_D$ are represented by $\Lambda^r_{\ZZ}(A_\lambda\otimes_{\ZZ}T_\bullet)$ and $\Lambda^r_{\ZZ}(A\otimes_{\ZZ}T_\bullet)$, respectively. Tensoring and forgetting transfers commute with filtered colimits, so
\[
\Lambda^r_{\ZZ}(A\otimes_{\ZZ}T_\bullet)\cong
\colim_\lambda \Lambda^r_{\ZZ}(A_\lambda\otimes_{\ZZ}T_\bullet).
\]
As in the symmetric-power argument above, applying \cite[XII, \S3.5]{bousfield1972homotopy} objectwise gives
\[
K(A,2n,n)_D\simeq
\hocolim_\lambda K(A_\lambda,2n,n)_D
\]
in $H_{\nis,\aff^1}(D_+)$. The finitely generated case and \Cref{lem:comparison-hocolimits} therefore prove the assertion for field coefficients.

For general coefficients, write $\theta_R$ for the comparison of $K(A,2n,n)_D$ with coefficients in $R$. The restriction functor $i_{\tr,\rad}$ is additive and preserves small homotopy colimits by the proof of \Cref{lem:comparison-hocolimits}. It therefore commutes with derived change of coefficients, as can be checked after forgetting the $R$-action using a simplicial free abelian resolution of $R$. Forgetting coefficients reflects isomorphisms. Together with change of coefficients for reduced motives and the naturality of the adjunction maps defining $\theta$, this gives
\[
\theta_{\ZZ}\otimes_{\ZZ}^{\LL}R\simeq\theta_R,
\]
where the tensor product denotes derived extension of coefficients.

View $\theta_{\ZZ}$ in the triangulated category $\dm^{\eff}_-(C,\ZZ)$ via the fully faithful functor
\[
H_{\nis,\aff^1}(\cor(C,\ZZ))\hookrightarrow \dm^{\eff}_-(C,\ZZ)
\]
of \cite[Theorem~1.15]{voevodsky2010motivic}, and let $E$ be its cone. The field case gives
\[
E\otimes_{\ZZ}^{\LL}\QQ=0,
\quad E\otimes_{\ZZ}^{\LL}\FF_p=0
\quad\text{for every prime }p.
\]
Here the scalar extensions are viewed as integral motives by forgetting coefficients. Since $E\otimes_{\ZZ}^{\LL}\FF_p$ is the cone of multiplication by $p$ on $E$, every prime acts invertibly on $E$. Thus $E$ is isomorphic to its rationalization $E\otimes_{\ZZ}^{\LL}\QQ$, which is zero. Full faithfulness now shows that $\theta_{\ZZ}$ is an isomorphism, and derived extension of coefficients proves the assertion for arbitrary $R$.
\end{proof}

\subsection{Betti realization and bistable operations}
We introduce the notation needed to state the main theorem of this section.

Let $k$ be a field of characteristic zero, and let $\ell$ be a prime. Let $\cA^{*,*}(k,\FF_\ell)$ denote the motivic Steenrod algebra constructed by Voevodsky in \cite[Section~11]{voevodsky2003reduced}.

Let $\sh(k)$ be the stable motivic homotopy category, obtained by stabilizing the motivic homotopy category $H_{\nis,\aff^1}((\sm/k)_+)\simeq H_{\nis,\aff^1}((\smqp/k)_+)$ with respect to $(\proj^1,\infty)\simeq S^{2,1}$ \cite[Section~2]{hoyois2015cobordism}. Its objects are called \textit{motivic spectra}. The category $\sh(k)$ is triangulated, with translation functor $\Sigma^{1,0}$. 

Let $\hz$ and $\hf_\ell$ be the motivic Eilenberg--MacLane spectra over $k$ representing integral and mod-$\ell$ motivic cohomology, respectively. They can be obtained by stabilizing the corresponding motivic Eilenberg--MacLane spaces $K(\ZZ,2n,n)$ and $K(\FF_\ell, 2n,n)$.

The Betti realization functor $\real: \sm/\CC\to \topcat$ extends to the motivic homotopy category, and induces a corresponding functor from motivic spectra over $\CC$ to the classical stable homotopy category.

\begin{theorem}\label{thm:bistable-operations}
Let $k$ be a field of characteristic zero and $\ell$ a prime.
\begin{enumerate}[(1)]
\item For $k=\CC$, complex realization gives natural equivalences
\[
\real\bigl(K(A,2n,n)_{\smqp/\CC}\bigr)
\simeq K(A,2n)
\]
for every abelian group $A$ and $n>0$, where the right-hand side is the classical Eilenberg--MacLane space.
\item For $k=\CC$, stabilization gives natural equivalences
\[
\real(\hz)\simeq \hz,\quad
\real(\hf_\ell)\simeq \hf_\ell,
\]
where the right-hand sides are the corresponding classical Eilenberg--MacLane spectra. These equivalences respect reduction and the Bockstein.
\item For $k=\CC$ and $A=\FF_\ell$, the comparisons in (1) and (2) are compatible with motivic and classical Steenrod operations.
\item Evaluation of operations on the universal classes induces an isomorphism of bigraded algebras
\[
\cA^{*,*}(k,\FF_\ell)\xrightarrow{\simeq}
\varprojlim_{n>0}\widetilde H^{*+2n,*+n}\bigl(K(\FF_\ell,2n,n)_{\sm/k},\FF_\ell\bigr).
\]
Here the transition maps are induced by $S^{2,1}$-suspension, and the algebra structure on the right is given by composition of operations.
\end{enumerate}
\end{theorem}
\begin{proof}
We follow Voevodsky's proofs of \cite[Corollary~3.48, Lemma~3.54, and Theorem~3.49]{voevodsky2010motivic} verbatim, except for the use of resolution of singularities. In these arguments, resolution of singularities is used only through Theorem~1.21 in the proof of Theorem~3.32, to compare the reduced motive of a restricted space with the restriction of its reduced motive. The assertions above require this comparison only for the Eilenberg--MacLane spaces $K(A,2n,n)_{\snqp/k}$, so \Cref{thm:comparison-eilenberg-maclane}, applied to $\smqp/k\hookrightarrow\snqp/k$, supplies the required replacement.
\end{proof}


\section{Thom obstructions}\label{sec:thom-obstructions}

Throughout this section we work over $\CC$ and fix a prime $\ell$. 
The plan of this section is as follows. First we lift the construction of the Brown--Peterson spectrum in \cite{brown1966spectrum} to motivic spectra, in such a way that the successive motivic $k$-invariants in the construction all have motivic bidegrees of the form $(2d+1,d)$. Then combining this with standard vanishing results for motivic cohomology, we show that all the Thom obstructions to representing the singular homology class of a complex algebraic cycle by a complex bordism class vanish.

\subsection{Recollections on the motivic Steenrod algebra}
Write $\cA=\cA^{*,*}(\CC,\FF_\ell)$ for the motivic Steenrod algebra over $\CC$ with $\FF_\ell$-coefficients. Choosing a primitive $\ell$-th root of unity gives an isomorphism
\[
H^{*,*}(\Spec(\CC),\FF_\ell)\simeq\FF_\ell[\tau],
\quad |\tau|=(0,1).
\]
We identify this coefficient ring with its image in $\cA$, which is \textit{central}.

As an $\FF_\ell[\tau]$-algebra, $\cA$ is generated by the Bockstein $\beta$ and the reduced powers $P^i$ $(i\ge0)$ \cite[Section~11]{voevodsky2003reduced}, whose bidegrees are
\[
|\beta|=(1,0),\quad |P^i|=(2i(\ell-1),i(\ell-1)).
\]
These operations satisfy $P^0=1$, $\beta^2=0$, and the motivic Adem relations \cite[Theorems~10.2--10.3]{voevodsky2003reduced}; their action on products is given by the Cartan formulas \cite[Section~9]{voevodsky2003reduced}. When $\ell=2$, we write $\Sq^{2i}=P^i$ and $\Sq^{2i+1}=\beta P^i$.

The structure of $\cA$ is almost identical to that of the classical Steenrod algebra. As an $\FF_\ell[\tau]$-module, $\cA$ admits the Milnor basis \cite[Section~13]{voevodsky2003reduced}:
\[
 \bigl\{Q_0^{\epsilon_0}Q_1^{\epsilon_1}\cdots P^R\bigr\}.
\]
Here $\epsilon_j\in\{0,1\}$, only finitely many $\epsilon_j$ are nonzero, and $R$ ranges over the set $\cR$ of sequences $(r_1,r_2,\ldots)$ of nonnegative integers with finite support. The $Q_j$ are the Milnor primitives and the $P^R$ are the even Milnor operations, with $Q_0=\beta$ and $P^{(n,0,\ldots)}=P^n$ \cite[Lemmas~13.1 and~13.5]{voevodsky2003reduced}. The antipode $c:\cA\to\cA$ satisfies $c(Q_j)=-Q_j$.

The bidegrees of the Milnor operations are given by
\[
 |Q_i|=(2\ell^i-1,\ell^i-1),
 \quad |P^R|=(2d(R),d(R)),
\]
in which
\[
 d(R)=\sum_{j\geq1}r_j(\ell^j-1),
 \quad
 l(R)=\sum_{j\geq1}r_j.
\]
For $U,R\in\cR$, write $U\leq R$ for componentwise inequality and $R-U$ for the componentwise difference. For $j\geq1$, let $\Delta_j$ denote the sequence having $1$ in position $j$ and zero elsewhere.

The following lemma collects facts from Voevodsky's structure theorem and Milnor basis calculations \cite[Theorem~12.6, Lemma~12.11, and Propositions~13.2 and~13.4]{voevodsky2003reduced}. When $\ell=2$, we use the vanishing of $\rho=[-1]\in H^{1,1}(\CC;\FF_2)$.

\begin{lemma}\label{lem:motivic-steenrod-structure}
\begin{enumerate}[(1)]
\item Let $\cA_0\subset\cA$ be the $\FF_\ell[\tau]$-subalgebra generated by the $Q_i$. Then $\cA_0$ is an exterior algebra over $\FF_\ell[\tau]$, and $\cA$ is free as a right $\cA_0$-module.

\item Let $\cA Q_0$ denote the left ideal generated by $Q_0$ and $(Q_0)$ the two-sided ideal it generates. The classes $P^R$ form an $\FF_\ell[\tau]$-basis of $\cA/(Q_0)$.

\item The following identities hold for $R\in\cR$ and $j\geq1$:
\begin{align*}
 Q_j&=Q_0c(P^{\Delta_j})-c(P^{\Delta_j})Q_0,\\
 Q_0c(P^R)&=
 \sum_{j:r_j>0}c(P^{R-\Delta_j})Q_0c(P^{\Delta_j})+aQ_0,
\end{align*}
for some $a\in\cA$. These are the motivic analogues of the identities in \cite[Lemma~2.6 and its proof, p.~151]{brown1966spectrum}.
\end{enumerate}
\end{lemma}

For $s\geq0$, let $V_s$ be the free bigraded $\ZZ$-module on the symbols $R\in\cR$ with $l(R)=s$, placing $R$ in bidegree $(2d(R),d(R))$. Define the left $\cA$-module
\[
 M_s=(\cA/\cA Q_0)\otimes_{\ZZ}V_s.
\]
For $s\geq1$, let $d_s:M_s\to M_{s-1}$ be determined $\cA$-linearly by
\[
 d_s(1\otimes R)
   =\sum_{j:r_j>0}Q_j\otimes(R-\Delta_j).
\]
These maps are well defined because $Q_0Q_j=-Q_jQ_0$. Each $d_s$ has bidegree $(1,0)$, since every summand $Q_j\otimes(R-\Delta_j)$ has bidegree $(2d(R)+1,d(R))$. Let $\alpha:M_0\to\cA/(Q_0)$ be the quotient map, so that $\alpha(1\otimes(0,0,\ldots))=1$.

\begin{proposition}\label{prop:motivic-bp-resolution}
The maps $d_s$ and $\alpha$ form an exact sequence
\[
 \cdots\longrightarrow M_s\xrightarrow{d_s}M_{s-1}
 \longrightarrow\cdots\longrightarrow M_0
 \xrightarrow{\alpha}\cA/(Q_0)\longrightarrow0.
\]
The differentials $d_s$ preserve motivic weight and raise cohomological degree by one; the augmentation $\alpha$ has bidegree $(0,0)$.
\end{proposition}

\begin{proof}
Let $\cA_0^+=(Q_0,Q_1,\ldots)$ be the augmentation ideal of $\cA_0$ over $\FF_\ell[\tau]$. Then $\cA\cA_0^+=(Q_0)$; cf.\ \cite[equation~(2.4), p.~151]{brown1966spectrum}. Put $B=\cA_0/\cA_0Q_0$, the exterior algebra over $\FF_\ell[\tau]$ on the $Q_i$ with $i>0$.

The trivial $B$-module $\FF_\ell[\tau]$ admits the standard free resolution
\[
 \cdots\to B\otimes_{\ZZ}V_s
 \to B\otimes_{\ZZ}V_{s-1}\to\cdots
 \to B\otimes_{\ZZ}V_0
 \to\FF_\ell[\tau]\to0,
\]
with differential given by the formula defining $d_s$; see \cite[Lemma~2.7 and equation~(2.8), pp.~151--152]{brown1966spectrum} (note that $\FF_\ell[\tau]$ is flat over $\FF_\ell$). View this as a sequence of $\cA_0$-modules and apply the exact functor $\cA\otimes_{\cA_0}(-)$. The resulting sequence is the one in the statement, since
\[
 \cA\otimes_{\cA_0}B\simeq\cA/\cA Q_0,
 \quad
 \cA\otimes_{\cA_0}\FF_\ell[\tau]\simeq\cA/(Q_0).
\]
The bidegrees of $d_s$ and $\alpha$ follow immediately from their definitions.
\end{proof}

\subsection{A motivic lift of the Brown--Peterson tower}

The goal of this subsection is to lift the original construction of the Brown--Peterson spectrum \cite{brown1966spectrum} to $\sh(\CC)$. The key lemma required for describing the inductive $k$-invariants and transgressions is the following.

\begin{lemma}[cf.\ {\cite[Lemma~3.1]{brown1966spectrum}}]\label{lem:bockstein-lifting}
Let $F\xrightarrow{i}E\xrightarrow{\pi}B$ be a fiber sequence of motivic spectra, and let
\[
 \bar{\btau}\colon H^{a,b}(F;\FF_\ell)\longrightarrow
 H^{a+1,b}(B;\FF_\ell)
\]
be its mod-$\ell$ transgression. If $v\in H^{a,b}(F;\FF_\ell)$ and $u\in H^{a+1,b}(B;\ZZ)$ satisfy $\bar{\btau}(v)=\bar u$, then there exists $w\in H^{a+1,b}(E;\ZZ)$ such that
\[
 \pi^*u=\ell w,\quad i^*w=\delta v.
\]
Here $\bar x$ denotes the mod-$\ell$ reduction of an integral class $x$, and $\delta\colon H^{a,b}(-;\FF_\ell)\to H^{a+1,b}(-;\ZZ)$ is the integral Bockstein associated with
\[
 0\longrightarrow\ZZ\xrightarrow{\ell}\ZZ
 \longrightarrow\FF_\ell\longrightarrow0.
\]
The mod-$\ell$ reduction of $\delta$ is the Milnor primitive $Q_0=\beta$ \cite[Lemma~13.5]{voevodsky2003reduced}; in particular, $Q_0\bar x=0$ for every integral class $x$.
\end{lemma}

\begin{proof}
For simplicity, put
\[
T=\Sigma^{a,b}\hz,
\quad \overline T=\Sigma^{a,b}\hf_\ell.
\]
Thus $v$ and $u$ are represented by maps $v:F\to\overline T$ and $u:B\to\Sigma^{1,0}T$. Let $\partial:B\to\Sigma^{1,0}F$ be the connecting map of the distinguished triangle associated with the given fiber sequence. We use the transgression convention
\[
\bar{\btau}(v)=(\Sigma^{1,0}v)\circ\partial.
\]
If $\rho:T\to\overline T$ denotes coefficient reduction, then $\bar u=(\Sigma^{1,0}\rho)\circ u$. The hypothesis therefore says that
\[
(\Sigma^{1,0}\rho)\circ u=(\Sigma^{1,0}v)\circ\partial.
\]

The coefficient sequence gives the distinguished triangle $T\xrightarrow{\ell}T\xrightarrow{\rho}\overline T\xrightarrow{\delta}\Sigma^{1,0}T$. Rotate the fiber triangle backwards once and the coefficient triangle forwards once. This gives the two rows of the diagram
\[
\begin{tikzcd}
\Sigma^{-1,0}B \ar[r,"-\Sigma^{-1,0}\partial"] \ar[d,"-\Sigma^{-1,0}u"']
& F \ar[r,"i"] \ar[d,"v"]
& E \ar[r,"\pi"] \ar[d,dashed,"w"]
& B \ar[d,"-u"] \\
T \ar[r,"\rho"]
& \overline T \ar[r,"\delta"]
& \Sigma^{1,0}T \ar[r,"-\ell"]
& \Sigma^{1,0}T.
\end{tikzcd}
\]
The left square commutes by the preceding equality. The morphism axiom for distinguished triangles therefore supplies the dashed map $w:E\to\Sigma^{1,0}T$ making the whole diagram commute; the last vertical map is the $(1,0)$-suspension of $-\Sigma^{-1,0}u$, hence is $-u$. The middle and right squares give
\[
w\circ i=\delta\circ v,
\quad (-\ell)\circ w=(-u)\circ\pi.
\]
The map $w$ represents a class in $H^{a+1,b}(E;\ZZ)$, and these equalities say exactly that $i^*w=\delta v$ and $\ell w=\pi^*u$.
\end{proof}

We now combine \Cref{prop:motivic-bp-resolution} and \Cref{lem:bockstein-lifting} to construct a motivic version of Brown--Peterson's tower.

For each $s\geq0$, let the Eilenberg--MacLane spectrum associated with the bigraded lattice $V_s$ be
\[
 K(V_s)=\prod_{l(R)=s}
 \Sigma^{2d(R),d(R)}\hz.
\]
Note that for every $N$ only finitely many sequences in the index set $\{R:l(R)=s\}$ satisfy $d(R)\le N$. Thus the canonical map from the coproduct to the product is an equivalence in $\sh(\CC)$ by \cite[proof of Proposition~5.5]{hoyois2015cobordism}:
\[
\bigoplus_{l(R)=s}\Sigma^{2d(R),d(R)}\hz
\xrightarrow{\simeq}K(V_s).
\]

Let $\alpha_R$ denote the universal integral cohomology class of the factor indexed by $R$; thus $\alpha_R\in H^{2d(R),d(R)}(K(V_s);\ZZ)$. The coefficient cofiber sequence and \Cref{thm:bistable-operations} give
\[
 H^{*,*}(\hz;\FF_\ell)
   \simeq \cA/\cA Q_0.
\]
Using the direct-sum description of $K(V_s)$ and the fact that each fixed bidegree receives contributions from only finitely many summands, we obtain
\[
 H^{*,*}(K(V_s);\FF_\ell)\simeq M_s.
\]

\begin{theorem}[{cf. \cite[Theorem~1.1]{brown1966spectrum}}]\label{thm:motivic-bp-construction}
There are motivic spectra $\bp_s^{\mathrm{mot}}$, $s\geq0$, and classes $\bone_s\in H^{0,0}(\bp_s^{\mathrm{mot}};\FF_\ell)$ satisfying the following conditions. 
\begin{enumerate}
	\item $\bp_0^{\mathrm{mot}}=K(V_0)=\hz$, with $\bone_0=\bar\alpha_{(0,0,\ldots)}$. 
	\item For $s>0$ there is a fiber sequence
\[
 K(V_s)\longrightarrow \bp_s^{\mathrm{mot}}
 \xrightarrow{\pi_s}\bp_{s-1}^{\mathrm{mot}},
\]
such that $\bone_s=\pi_s^*\bone_{s-1}$ and for every $R$ with $l(R)=s$ one has
\[
 \ell^{s-1}\bk_R^{s-1}=\delta c(P^R)\bone_{s-1}.
\]
Here, if $\btau_s$ denotes the integral transgression, then $\bk_R^{s-1}$ is the class given by
\[
\bk_R^{s-1}:=\btau_s(\alpha_R)\in H^{2d(R)+1,d(R)}(\bp_{s-1}^{\mathrm{mot}};\ZZ).
\]
	\item For $s>0$, the map
\[
 \cA\longrightarrow H^{*,*}(\bp_s^{\mathrm{mot}};\FF_\ell),
 \quad a\longmapsto a\bone_s,
\]
has kernel $(Q_0)$, and
\[
 H^{p,q}(\bp_s^{\mathrm{mot}};\FF_\ell)
   =\bigl((\cA/(Q_0))\bone_s\bigr)^{p,q}
 \quad\text{for }p<2(s+1)(\ell-1).
\]
	\item Betti realization identifies this tower with the classical integral Brown--Peterson tower $\bp_s^{\mathrm{top}}$ of \cite[Theorem~1.1]{brown1966spectrum}:
\[
 \real(\bp_s^{\mathrm{mot}})
 \simeq \bp_s^{\mathrm{top}},
\]
compatibly with the tower maps.
\end{enumerate}
\end{theorem}

\begin{proof}
We follow Brown--Peterson's construction \cite[Theorem~1.1]{brown1966spectrum} verbatim. The induction uses \Cref{lem:bockstein-lifting} to construct the successive integral $k$-invariants, together with the algebraic identities of \Cref{lem:motivic-steenrod-structure} and the exact resolution of \Cref{prop:motivic-bp-resolution}. Keeping track of the bidegrees of the Milnor operations, with $|\delta|=(1,0)$, gives $|\bk_R^{s-1}|=(2d(R)+1,d(R))$; the transgression formula and cohomology statements follow from the same induction.

Betti realization preserves coproducts and sends each summand to $\Sigma^{2d(R)}\hz$ by \Cref{thm:bistable-operations}. Using the direct-sum description of $K(V_s)$ above, we therefore obtain, in the classical stable homotopy category,
\[
\real K(V_s)
\simeq\bigoplus_{l(R)=s}\Sigma^{2d(R)}\hz
\simeq\prod_{l(R)=s}\Sigma^{2d(R)}\hz,
\]
which is the corresponding layer of the classical Brown--Peterson tower.

By \Cref{thm:bistable-operations}, Betti realization also respects the coefficient sequence and Milnor operations used in the construction. Together with preservation of fiber sequences, this identifies the two towers by induction.
\end{proof}

\begin{remark}
	The $\ell$-localization of the homotopy limit $\bp^{\mathrm{mot}}=\holim_s \bp_s^{\mathrm{mot}}$ is equivalent to the image of the Quillen idempotent of $\mgl_{(\ell)}$. Indeed the Thom class of this summand gives a map to $\hz_{(\ell)}$, and this map can be lifted to $\bp^\mathrm{mot}_{(\ell)}$ by the construction of $\bp^\mathrm{mot}$. Finally, one checks that the map is an isomorphism on cohomology. However we do not need this comparison for the current work.
\end{remark}

\subsection{Vanishing of Thom obstructions}

\begin{theorem}\label{thm:cycle-bp-lifting}
Let $M$ be a smooth equidimensional quasi-projective complex variety, and let $X\subset M$ be closed. For every $q\ge0$, each class $x_0\in H_X^{2q,q}(M;\ZZ)$ admits compatible lifts $x_s\in (\bp_s^{\mathrm{mot}})_X^{2q,q}(M)$ through the motivic Brown--Peterson tower for all $s\ge0$.
\end{theorem}
Here $H_X^{*,*}(M;\ZZ)$ denotes the reduced motivic cohomology of $M/(M-X)$, and we use the analogous notation for the other theories.
\begin{proof}
Write $m=\dim M$. For $b\ge0$ and $a>2b$, the standard comparison with higher Chow groups with supports \cite[Theorem~19.1 and Properties~17.4(2)]{mazza2006lectures} gives
\[
H_X^{a,b}(M;\ZZ)\simeq\CH_{m-b}(X,2b-a)=0,
\]
where the Chow groups are indexed by dimension. The vanishing follows from $2b-a<0$, since Bloch's cycle complex is concentrated in nonnegative homological degrees \cite[Definition~17.1]{mazza2006lectures}.

By \Cref{thm:motivic-bp-construction}, the successive lifting obstructions have components in the groups
\[
H_X^{2q+2d(R)+1,q+d(R)}(M;\ZZ),
\]
which vanish for degree reasons.
\end{proof}



\begin{theorem}
Let $X$ be a projective complex variety and let $\sigma\in H_{2d}(X(\CC);\ZZ)$ be the cycle class of an algebraic cycle. Then there exist a closed stably almost complex manifold $Y$ of real dimension $2d$ and a continuous map $f: Y\to X(\CC)$ such that $f_*[Y]=\sigma$.

In particular, if $X$ is pure-dimensional, there exists a closed stably complex manifold $\widetilde{X}$ and a continuous map $\pi:\widetilde{X}\to X(\CC)$ such that $\pi_*[\widetilde{X}]=[X]$.
\end{theorem}
\begin{proof}
Choose an embedding $X\hookrightarrow M$ into a smooth projective variety $M$, and put $q=\dim M-d$. By assumption, the Alexander dual $x_0\in\widetilde H^{2q}_{X(\CC)}(M(\CC);\ZZ)$ of $\sigma$ is the realization of a class in $H^{2q,q}_X(M;\ZZ)$. By \Cref{thm:cycle-bp-lifting} and Betti realization, $x_0$ admits compatible lifts through the classical tower $\bp_s^{\mathrm{top}}$. Thus $x_0$ lies in the image of
\[
\varprojlim_s(\bp_s^{\mathrm{top}})^{2q}_{X(\CC)}(M(\CC))
\longrightarrow\widetilde H^{2q}_{X(\CC)}(M(\CC);\ZZ).
\]
Note that $M(\CC)/(M(\CC)-X(\CC))$ has the homotopy type of a finite CW complex.

Applying this argument for every prime $\ell$ and using \cite[Theorem~4]{quillen1969formal}, we see that the image of $x_0$ in the cokernel of
\[
\widetilde{\MU}^{2q}_{X(\CC)}(M(\CC))
\longrightarrow\widetilde H^{2q}_{X(\CC)}(M(\CC);\ZZ)
\]
vanishes after localization at every prime and hence is zero. Thus $x_0$ lifts to complex cobordism. By duality, using the canonical complex orientation of $M(\CC)$, the class $\sigma$ lifts to complex bordism. The standard Pontryagin--Thom argument then shows that $\sigma$ is represented by a map $f:Y\to X(\CC)$ as asserted.
\end{proof}

\bibliographystyle{plain}
\bibliography{ref}

\Addresses
\end{document}